\documentclass[11pt]{amsart}

\calclayout

\usepackage{graphicx}
\usepackage{amssymb, amsthm,mathrsfs}
\usepackage[all]{xy}
\usepackage[dvipsnames]{xcolor}
\usepackage{comment}
\usepackage{pinlabel}
\usepackage[inline]{enumitem}
\usepackage{todonotes}
\usepackage{thm-restate}
\usepackage{verbatim}
\usepackage{tikz}
\usepackage{soul}
\usepackage{hyperref}
\usepackage[nameinlink]{cleveref}
\usepackage{cite}

\hypersetup{colorlinks=true,linkcolor=teal,citecolor=purple}

\newtheorem{theorem}{Theorem}[section]
\newtheorem{lemma}[theorem]{Lemma}
\newtheorem{proposition}[theorem]{Proposition}

\theoremstyle{definition}
\newtheorem{definition}[theorem]{Definition}
\newtheorem{remark}[theorem]{Remark}

\newcommand{\R}{{\ensuremath{\mathbb{R}}}}

\newcommand{\Z}{{\ensuremath{\mathbb{Z}}}}

\newcommand{\CS}{{\rm CS}}
\newcommand{\mfs}{\mathfrak{s}}

\newcommand{\spinc}{\text{Spin}^c}
\DeclareMathOperator{\tr}{tr}
\DeclareMathOperator{\ind}{ind}

\newcommand{\PHS}{P}

\theoremstyle{introtheorem}
\newtheorem{introtheorem}{Theorem}
\theoremstyle{introcor}

\theoremstyle{introrem}

\author[Aliakbar Daemi]{Aliakbar Daemi}
\address{Department of Mathematics and Statistics, Washington University in St. Louis}
\email{adaemi@wustl.edu}

\author[Tye Lidman]{Tye Lidman}
\address{Department of Mathematics, North Carolina State University}
\email{tlid@math.ncsu.edu}

\author[Mike Miller Eismeier]{Mike Miller Eismeier}
\address{Department of Mathematics, Columbia University}
\email{smm2344@columbia.edu}

\title{3-manifolds without any embedding in symplectic 4-manifolds}

\begin{document}
\maketitle

\begin{abstract}
We show that there exist infinitely many closed 3-manifolds that do not embed in closed symplectic 4-manifolds, disproving a conjecture of Etnyre-Min-Mukherjee.  To do this, we construct L-spaces that cannot bound positive or negative definite manifolds.  The arguments use Heegaard Floer correction terms and instanton moduli spaces.
\end{abstract}

\begin{introtheorem}\label{thm:symplectic}
There exist infinitely many rational homology spheres which cannot embed in a closed symplectic 4-manifold.  
\end{introtheorem}

The family of manifolds we use are particular connected sums of elliptic manifolds.  Let $\PHS$ denote the Poincar\'e homology sphere oriented as the boundary of the negative definite $E_8$ plumbing, equivalently $-1$-surgery on the left-handed trefoil.  Let $O$ denote the ``first'' octahedral manifold with Seifert invariants $(-2; 1/2, 2/3, 3/4)$ oriented as the boundary of the negative definite $E_7$ plumbing, equivalently $-2$-surgery on the left-handed trefoil (see for example \cite[Theorem 2, Equation 2]{Doig}). The manifolds in the theorem are those of the form $m\PHS \# {-} k O$ with $m\geq 1$ and $k > 8m$.  This answers the conjecture of Etnyre-Min-Mukherjee from \cite[p.6]{EtnyreMinMukherjee} in the negative.  (Note that this is stronger than saying that the manifolds are not symplectically fillable, since a separating 3-manifold may sit in a symplectic 4-manifold in a way which is not compatible with any contact structure on the 3-manifold.)  

The rational homology spheres above are L-spaces, since they are connected sums of elliptic manifolds \cite[Section 2]{HFKLens}.  It is shown in \cite{Mukherjee} that if an L-space embeds in a symplectic 4-manifold, then it must bound a definite 4-manifold.  Hence, we are able to prove Theorem~\ref{thm:symplectic} by proving:

\begin{introtheorem}\label{thm:main}
For any pair of integers $k$ and $m$ with $m\geq 1$ and $k > 8m$, the manifolds $m\PHS \#{-}k O$ are L-spaces which cannot bound positive- or negative-definite 4-manifolds.
\end{introtheorem}

The argument has two steps.  In Section~\ref{sec:d-invariants} we use the Heegaard Floer correction terms to obstruct the negative-definite manifolds.  In Section~\ref{sec:chern-simons} we use Chern--Simons invariants and ASD moduli spaces to obstruct the positive-definite manifolds.  This section is where the real importance of our choice of the manifolds $O$ and $P$ appears.  In particular, because these manifolds have finite fundamental group, the values of the Chern-Simons functional are bounded in terms of the order of $\pi_1$.  This then greatly affects the structure of the moduli space of ASD connections on a definite manifold with boundary a sum of $P$'s and $O$'s.  The two essential properties of $O$ used in the proof are that $|\pi_1(O)| < |\pi_1(P)|$ and all of the $d$-invariants of $O$ are strictly negative. We also use that $H_1(O)$ is 2-torsion to simplify the discussion, but this is not an essential point.

Examples of 3-manifolds that do not bound any definite 4-manifold were previously given in \cite{NST:def-IHS,GL:def-RHS}. In \cite{NST:def-IHS}, a filtration of instanton Floer homology given by the Chern--Simons functional is used to construct integer homology spheres without any positive- or negative-definite 4-manifold filling. In \cite{GL:def-RHS}, the Heegaard Floer correction terms are used to construct examples of rational homology spheres that bound no definite 4-manifold. 

\section*{Acknowledgements} TL and MME thank the Department of Mathematics and Statistics at Wash U for their hospitality during their visits. The authors appreciate the helpful comments of anonymous referees on previous drafts of this article. AD was partially supported by NSF Grants DMS-1812033, DMS-2208181 and NSF FRG Grant DMS-1952762. TL was partially supported by NSF grant DMS-2105469 and a Sloan fellowship. MME was partially supported by NSF FRG Grant DMS-1952762.\\

\section{The $d$-invariant argument}\label{sec:d-invariants}
In this section, we use Heegaard Floer $d$-invariants \cite{OzSz:d-inv} to obstruct the manifolds $m \PHS \#{-}k O$ from bounding negative definite manifolds for suitable positive values of $k, m$.

\begin{proposition}\label{prop:d-neg}
Let $k, m > 0$.  The manifold $m\PHS \# {-} k O$ cannot bound a negative-definite 4-manifold for $k > 8m$.  
\end{proposition}

Before proving this, we need to compute the Heegaard Floer $d$-invariants of $O$.  
\begin{lemma}\label{lem:d-comp}
For a choice of labelling, the two $\spinc$ structures on $O$, $\mfs_0, \mfs_1$, satisfy 
\[
d(-O,\mfs_0) = -7/4, \quad d(-O,\mfs_1) = -1/4.  
\]
\end{lemma}
\begin{proof}
There are several ways to compute the $d$-invariants of $O$.  We opt for the surgery description of $-O$ as $S^3_2(T_{2,3})$.  By \cite[Theorem 6.1]{OwensStrle} (and the formulas following), there is a labeling of the $\spinc$ structures such that 
\[
d(S^3_2(T_{2,3}), \mfs_0) = \frac{1}{4} - 2t_0(T_{2,3}), \quad d(S^3_2(T_{2,3}), \mfs_1) = -\frac{1}{4} - 2t_1(T_{2,3}),
\]
where $t_i(K)$ denotes the $i$th torsion coefficient, $\sum_{j \geq 1} j a_{|i|+j}$, and $a_{k}$ 
is the $k$th coefficient of the symmetrized Alexander polynomial.  The result now follows, since $t_0(T_{2,3}) = 1$ and $t_1(T_{2,3}) = 0$ (see \cite[Equations 2 and 3]{OwensStrle}).  
\end{proof}

\begin{proof}[Proof of Proposition \ref{prop:d-neg}]
Suppose that $m\PHS \# {-} k O$ bounds a negative-definite $4$-manifold. Then \cite[Proposition 5.2]{OwensStrle} implies that 
\[
\max_{\mfs} d(m\PHS \# {-} kO, \mfs) \geq \begin{cases} 0 & k \text{ even}, \\ \frac{1}{4} & k \text{ odd}. \end{cases}
\]
In particular, 
\begin{equation}\label{eq:d-inequality}
\max_{\mfs} d(m\PHS \# {-} kO, \mfs) \geq 0. 
\end{equation} 
Recall that $d$-invariants are additive under connected sum and that $d(\PHS) = 2$.  We thus have from Lemma~\ref{lem:d-comp} 
\[
\max_{\mfs} d(m\PHS \# {-} kO, \mfs) = 2m - \frac{k}{4}. 
\]
It follows that for $k > 8m$, \eqref{eq:d-inequality} is violated.
\end{proof}

\begin{remark}
It seems likely that Proposition~\ref{prop:d-neg} can also be proved using Donaldson's diagonalizability theorem and lattice techniques. We anticipate that the assumption $k > 8m$ can be relaxed somewhat using refinements of Fr\o yshov's instanton $h$-invariant for rational homology spheres.
\end{remark}

\section{The Chern--Simons argument}\label{sec:chern-simons}
In this section, we use the instanton moduli spaces to obstruct $m\PHS \# {-}kO$ from bounding a positive-definite manifold, complementary to the results in Proposition~\ref{prop:d-neg}. 

\begin{proposition}\label{prop:CS-bound}
Let $m$ and $k$ be integers with $m > 0$. Then the manifold $m\PHS \# {-}k O$ cannot bound a positive-definite 4-manifold.  
\end{proposition}

To prove this claim, we consider moduli spaces of $SU(2)$-instantons. We first provide a sketch of the argument for non-experts. In Section \ref{sec:moduli}, we discuss generalities about instanton moduli spaces; this section consists of mostly standard facts about such moduli spaces mainly aimed at a non-expert reader who needs a quick review of the relevant background. In Section \ref{sec:details}, we expand on the proof of Proposition \ref{prop:CS-bound} sketched below. 

\begin{proof}[Sketch of proof.] 
Suppose $X$ is a positive-definite compact 4-manifold with boundary $m\PHS \# {-}kO$. Reverse the orientation to obtain a negative-definite manifold with boundary ${-}m\PHS \# kO$. By attaching pairs-of-pants cobordisms $Y \# Y' \to Y \sqcup Y'$ (equivalently, attaching $3$-handles along the connected-sum spheres in the boundary) we obtain a cobordism $W_0: \PHS \to \sqcup_{m-1} {-}\PHS \sqcup_{|k|} \pm O,$ where the sign on $\pm O$ is the sign of $k$ (which we do not assume positive). Let $W$ be the result of performing surgery along a family of embedded loops in $W_0$ which give a basis for $H_1(W_0; \Bbb Q)$. A Mayer-Vietoris argument implies that $b_1(W) = 0$, and that the intersection form is unchanged by addition of $3$-handles and surgery on loops which are nontrivial in rational homology, so that $W$ remains negative-definite. 

We will obtain a contradiction by considering an orientable 1-dimensional moduli space $M$ of instantons on this cobordism $W$ (more precisely, we attach cylindrical ends to $W$ and consider a perturbation of the ASD equation).  Our contradiction will come from showing that the number of ends, counted with sign, is non-zero.  The ends of this moduli space $M$ correspond to gluing instantons on $W$ to instantons on the incoming end $\Bbb R \times P$ or the outgoing ends $\Bbb R \times {\pm} O$ (the sign coinciding with the sign of $k$) and $\Bbb R \times {-}\PHS$. 

We first construct a family of ends of $M$ by gluing a particular instanton on $\Bbb R \times \PHS$ to the reducible flat connections over $W$. These reducibles are determined by $H_1(W;\Z)$, and are isolated and well-behaved with respect to gluing because $b_1(W) = 0$ and $b^+(W) = 0$, respectively, as discussed in the proof of Lemma \ref{lemma:end-counting}(iii). This construction produces as many ends of $M$ as there are elements of $H_1(W;\Z) / H_1( \partial W; \Z)$, and they are all oriented in the same direction. 

We use \emph{topological energy} $\kappa(A)$ of instantons to establish that these are the only ends.  The moduli space $M$ is the moduli space of instantons with topological energy equal to $\frac{1}{120}$.  In general, topological energy is non-negative, additive under gluing of instantons, and multiplicative under passing to covering spaces. An instanton $A$ on $W$ determines flat connections $\alpha$ and $\alpha'$ on the incoming and outgoing boundary components of $W$, and the topological energy $\kappa(A)$ is equal modulo $\Z$ to the difference of \emph{Chern--Simons invariants} $\textrm{CS}(\alpha)-\textrm{CS}(\alpha') \in \Bbb R/\Z$.

There are two key points. \begin{itemize}
\item The instanton on $\Bbb R \times P$ used above has $\kappa(A) = \frac{1}{120}$, and the reducible flat connections on $W$ have $\kappa(A) = 0$. By additivity of energy, the instantons we constructed on $W$ above have energy $\kappa = \frac{1}{120}$, as do all other instantons in the moduli space $M$. All other instantons on $\Bbb R \times P$ have larger energy; when glued to instantons on $W$ they produce instantons of energy larger than $\frac{1}{120}$, which do not lie in $M$.

\item All instantons on $\Bbb R \times \pm O$ have $\kappa(A) \ge \frac{1}{48}$, so also do not contribute to ends of $M$. Here we use the relation to the Chern--Simons invariant: for any flat connection $\alpha$ on $O$, we have $48\textrm{CS}(\alpha) \equiv 0 \in \Bbb R/\Z$. This follows because the universal cover of $O$ is the 3-sphere, where the Chern--Simons invariant of a flat connection is zero; the Chern--Simons invariant is multiplicative under covers and $|\pi_1(O)| = 48$. 
\end{itemize}

Because every end of $M$ is constructed by the gluing procedure (gluing an instanton on a cylindrical end to one on $W$), the only ends are those initially described. In particular, $M$ is a noncompact oriented 1-manifold without boundary with finitely many ends. This gives the desired contradiction: the signed count of ends is $\pm |H_1(W;\Z) / H_1( \partial W; \Z)| \ne 0$, but a noncompact oriented 1-manifold without boundary has zero ends, counted with sign.
\end{proof}

\begin{remark}
	Proposition \ref{prop:CS-bound} holds more generally for any closed oriented 3-manifold $Y$ with $|\pi_1(Y)| < 120$ in 
	place of $O$, with a similar proof. 
\end{remark}

\begin{remark}
	Using moduli spaces of $SU(2)$-instantons to study negative definite smooth closed 4-manifolds
	goes back to Donaldson's groundbreaking work \cite{Don:neg-def-gauge}. 
	Here we use an energy argument to analyze boundary
	components of a 1-dimensional moduli space of $SU(2)$-instantons. Similar strategies appear, for example, in 
	\cite{FS:pseudofree, Fur:hom-cob,FS:HFSF,KH:Wh-dble,P:ind-br}. Another key tool 
	in the study of negative definite 4-manifolds with integer homology sphere boundary
	is Fr\o yshov's invariant $h$ of \cite{Fro:h-inv}.
	In fact, the $d$-invariant used in the previous section is the Heegaard Floer analogue of Fr\o yshov's invariant, and it is expected that $d = 2h$.   
	Topological energy is employed in \cite{D:CS-cob,NST:def-IHS} 
	to construct refinements of Fr\o yshov's invariant. We expect that the invariants of \cite{D:CS-cob,NST:def-IHS} generalize to invariants of rational homology spheres using the results of \cite{mme:eq-inst,DME1}, and that the argument above can be recast in that language.
\end{remark}

\subsection{Properties of instanton moduli spaces}\label{sec:moduli}
Suppose $W$ is a compact 4-manifold with boundary. Choose a metric on $W$ which is cylindrical (identical to the product metric) in a collar neighborhood of the boundary. We will consider instantons on the complete Riemannian manifold $W \cup_{\partial W} [0,\infty) \times \partial W$, where we have attached infinite cylindrical ends. By an abuse of notation, we ignore this subtlety and refer by the same name $W$ to both the compact manifold $W$ and the version with cylindrical ends. By partitioning $\partial W$ into a set of \emph{incoming ends} and \emph{outgoing ends}, we may write that $W: Y \to Y'$ is a cobordism, where $\partial W = Y' \sqcup {-}Y$. 

We are interested in $SU(2)$-connections $A$ on the trivial bundle over $W$ which are asymptotic to a flat connection $\alpha$ over $Y$ and a flat connection $\alpha'$ over $Y'$, and for which the {\it topological energy}, defined as
\[
  \kappa(A):=\frac{1}{8\pi^2}\int_W \tr(F_A\wedge F_A),
\]
is finite. This quantity is constant with respect to continuous deformations of $A$, and its mod $\Z$ value is equal to $\CS(\alpha) - \CS(\alpha')$. (Taking $Y'$ to be empty, this serves as a definition of $\CS(\alpha)$.) Further, as discussed in \cite[Section 3.2]{donaldson-book}, $\kappa$ determines the deformation class of $A$ through connections with flat limits $\alpha, \alpha'$.

We are more specifically interested in those finite-energy connections $A$ on $W$ which satisfy the ASD equation
\[
  F^+(A)=0
\]
with respect to the metric on $W$. Any such connection satisfying the ASD equation is also called an {\it instanton}. Any instanton $A$ satisfies 
\[\kappa(A) = \frac{1}{8\pi^2}|\!|F_A|\!|_{L^2}^2.\] 
In particular, $\kappa(A)\geq 0$, and $\kappa(A)=0$ if and only if $A$ is flat.

There is an infinite-dimensional space of instantons, because the ASD equation is invariant under an infinite-dimensional symmetry group. Define the {\it gauge group} to be the space of all maps $u:W \to SU(2)$ which are asymptotic to a map $v:Y\to SU(2)$ on the incoming end and a map $v':Y'\to SU(2)$ on the outgoing end, regarded as automorphisms of the trivial $SU(2)$-bundle. Then we may pull back any connection $A$ as above with respect to $u$ to obtain a connection $u^*A$ which has the same topological energy as $A$ and is asymptotic to $v^*\alpha$ and $(v')^*\alpha'$ on the incoming and the outgoing ends of $W$. 

The automorphisms $u=\pm I$ act trivially on $A$, and $A$ is called {\it irreducible} if these are the only elements of the stabilizer $\Gamma_A$ of $A$ under the action of the gauge group. The other possibilities for the isomorphism type of $\Gamma_A$ are $U(1)$ and $SU(2)$ where $A$ is called respectively an {\it abelian} and a {\it central} connection. For instance, the trivial connection is a central connection. We use similar terminology to define the three types of connections on 3-manifolds. 

Now for any non-negative real number \[\kappa \equiv \CS(\alpha) - \CS(\alpha') \mod \Z,\] let $M_\kappa(W;\alpha,\alpha')$ denote the \emph{moduli space of instantons}, the set of gauge equivalence classes of instantons $A$ on $W$ with topological energy $\kappa$ which are asymptotic to $\alpha$ along the incoming ends and $\alpha'$ along the outgoing ends. As is recalled below, these moduli spaces are finite-dimensional because they are locally modeled on the solution set to an elliptic equation. 

A special case of interest is when $W=\R\times Y$ for a connected 3-manifold $Y$. For any flat connections $\alpha, \beta$ on $Y$ and any $\kappa$ as above, we have a moduli space $M_\kappa(\R\times Y;\alpha,\beta)$. Translation along the $\R$ factor determines an action of $\R$ on this moduli space. This action is free if $\kappa$ is positive, and the quotient in this case is denoted by $\breve M_\kappa(Y;\alpha,\beta)$. The only instantons with $\kappa = 0$ are constant trajectories, which we exclude by defining $\breve M_\kappa(Y; \alpha, \beta) = \varnothing$. These moduli spaces $\breve M_\kappa(Y;\alpha, \beta)$ consist of irreducible instantons (see, for example, \cite[Proposition 4.14]{mme:eq-inst}). 

Reducible flat connections and instantons complicate the analysis of instanton moduli spaces. %, for three closely related reasons. First, the moduli spaces $M_\kappa(W)$ are not expected to be smooth manifolds in a neighborhood of reducible connections, but rather have local models given by quotient singularities, quotients of Euclidean spaces by actions of $\Gamma_A/\{\pm 1\}$. Second, there are index obstructions to achieving `transversality' at reducible connections, so that the local models might be even more complicated than those outlined above. Third, the gluing analysis of the ends of $M_\kappa(W)$ must be done by a case analysis depending on the reducibility type of the connections involved, related to the notion of `gluing factors'. 
We restrict attention to a special class of 3-manifolds which are general enough to prove the main theorem, but simple enough to avoid some of the difficulties of the reducibles. 

\begin{definition}
A {\it two-torsion homology sphere} is a closed, connected, oriented 3-manifold $Y$ for which $H_1(Y;\Z)$ consists only of two-torsion; that is, $H_1(Y;\Z) \cong (\Z/2)^a$ for some $a$. 
\end{definition}

Flat connections up to gauge equivalence are in bijection with homomorphisms $\pi_1(Y) \to SU(2)$ up to conjugacy. Reducible flat connections are those whose image is conjugate to a subgroup of $U(1)$, and are classified by homomorphisms $H_1(Y) \to S^1$ up to complex conjugation; the central flat connections are those with image in the center $\{\pm 1\}$, and are classified by homomorphisms $H_1(Y) \to \{\pm 1\}$. Pontryagin duality implies that $Y$ is a two-torsion homology sphere if and only if all flat connections on $Y$ are irreducible or central. 

The 4-manifolds we work with will be cobordisms $W: Y \to Y'$, where $Y$ and $Y'$ are disjoint unions of two-torsion homology spheres, and for which $b_1(W) = b^+(W) = 0$. The first assumption will guarantee that there are finitely many reducibles, while the second assumption will guarantee that reducibles are cut out transversely in the reducible locus. 

We will need an enumeration of these reducibles, in a special case. 
\begin{lemma}\label{lemma:count-reds}
Supposing $W: Y \to Y'$ is a cobordism with $b_1(W) = 0$, write $a$ for the number of abelian flat connections on $W$ which are asymptotically trivial, and write $z$ for the number of central flat connections on $W$ which are asymptotically trivial. Then $z+2a$ is equal to the cardinality of $H_1(W;\Z)/i_*H_1(\partial W;\Z)$, where $i_*$ is the homomorphism induced by inclusion.
\end{lemma}
\begin{proof}
Reducible flat $SU(2)$-connections on $W$ up to gauge correspond bijectively to homomorphisms $H_1(W) \to S^1 \subset SU(2)$ up to complex conjugation. The asymptotically trivial flat connections correspond to those homomorphisms which vanish on the boundary, or equivalently homomorphisms $H_1(W)/i_* H_1(\partial W) \to S^1$ up to complex conjugation. The central flat connections correspond to the fixed points of complex conjugation, while the abelian flat connections correspond to its free orbits. Thus $z+2a$ can be identified with the number of homomorphisms $H_1(W)/i_* H_1(\partial W) \to S^1$. 

Because $b_1(W) = 0$, the domain is a finite group, and Pontryagin duality non-canonically identifies the set of such homomorphisms with $H_1(W)/i_* H_1(\partial W)$ itself.
\end{proof}

In general, it need not be the case that the flat connections on $Y$ are cut out transversely by the equation $F_\alpha = 0$. Even if they are, it need not be the case that each $\breve M_\kappa(Y; \alpha, \beta)$ is cut out transversely. When both of these properties hold, we say that $Y$ is \emph{regular}. 

Henceforth, suppose that $W: Y \to Y'$ is a cobordism between disjoint unions of two-torsion regular two-torsion homology spheres. 

The local behavior of the moduli space around any instanton $A$ is governed by the {\it ASD operator} $D_A:=d_A^+\oplus d_A^*: \Omega^1 \to \Omega^+ \oplus \Omega^0$, which (on appropriate Sobolev completions) is a Fredholm operator obtained as a combination of the {\it Coulomb gauge} condition and linearizing the instanton equation at $A$ \cite[Chapter 3]{donaldson-book}. When at least one of $A$'s flat limits is reducible, these should be defined on appropriate \emph{weighted} Sobolev completions, as in \cite[Chapter 3.3.3]{donaldson-book}.

If $d_A^+$ is surjective, a neighborhood of $A$ in this moduli space is modeled on $\ker(D_A)/\Gamma_A$, where $\{\pm 1\} \subset \Gamma_A$ acts trivially. In particular, if $A$ is irreducible and $D_A$ is surjective, a neighborhood of $A$ is locally Euclidean; in this case we say that $A$ is cut out transversely. However, it is rarely the case that $D_A$ is surjective for all instantons $A$. This necessitates the introduction of a {\it holonomy perturbation} $\pi$ of the ASD equation. The precise definition of these holonomy perturbations is given in \cite[Definition 4.2]{mme:eq-inst}, following the presentation of \cite[Section 3]{K-HigherRank}.\footnote{Another approach to holonomy perturbations is given in \cite[Section 3.2]{YAFT} and \cite[Chapter 5.4]{donaldson-book}. These perturbation schemes are insufficient in our setting: they vanish automatically on flat connections which are central on the ends, which feature in our argument in an essential way. The former authors do not need to contend with this, as they restrict attention to admissible bundles (which support no reducible connections) and the latter author does not discuss in detail the perturbation scheme on cobordisms.} While the latter article focuses on the case of compact 4-manifolds without boundary, its definition of holonomy perturbation remains well-behaved so long as one demands they vanish on a chosen neighborhood of infinity in $W$. 

A standard argument shows that for a generic choice of $\pi$, the moduli space of solutions to
\begin{equation}\label{per-ASD}
  F^+(A)+\pi(A)=0
\end{equation}
is cut out transversely away from reducible elements of the moduli space \cite[Theorem 4.37]{mme:eq-inst}. We can pick $\pi$ so that the $L^2$ norm of $\pi(A)$ is less than a given positive constant, the perturbation $\pi(A)$ vanishes for any reducible connection $A$, and $\pi(A)$ vanishes outside of a fixed compact submanifold of $W$ for any connection $A$. 

When performing a gluing analysis of the ends of $\breve M^\pi_\kappa(W;\alpha, \alpha')$, we will also need to assume the perturbation is well-behaved on the reducible flat connections $\Lambda$ on $W$. This is somewhat more difficult. First, $D_\Lambda$ will not be surjective, as the cokernel of $d_A^*$ can be identified with the Lie algebra of the stabilizer $\Gamma_A$. 

Any reducible flat connection $\Lambda$ is asymptotically central (as the ends of $W$ are two-torsion homology spheres), so the index calculation of \cite[Proposition 4.26]{mme:eq-inst} drastically simplifies: the `$\rho$-invariant' and `signature data' terms are automatically zero for central connections. Because $\Lambda$ is flat, the characteristic class term is also trivial, so that the formula simply gives $\ind(D_\Lambda) = -3$.

Ideally, we would ensure that $d_\Lambda^+$ is surjective after perturbation, but this is sometimes impossible. When $\Lambda$ is central, the cokernel of $d_A^+$ has dimension $b^+(W) = 0$, and there is no issue. When $\Lambda$ is an abelian flat connection, the map $d_\Lambda^*$ has cokernel of rank $1$, so $d_\Lambda^+$ cannot be surjective: at best it can be injective, with cokernel of rank $2$. The following lemma asserts that this best case scenario can be always achieved.

\begin{lemma}\label{flat-W-pert}Suppose $W: Y \to Y'$ is a cobordism between a disjoint union of regular two-torsion homology spheres, and suppose that $b_1(W) = b^+(W) = 0$. For any $\epsilon > 0$ there is a holonomy perturbation $\pi$ for $W$ such that $\pi(A)$ vanishes for $A$ reducible, we have $\|\pi(A)\|_{L^2} < \epsilon$ for all $A$, and the following hold.
	\begin{itemize}
	\item[(i)] All irreducible solutions $A$ to the perturbed ASD equations have surjective ASD operator, so that the irreducible part of $M^\pi_\kappa(W;\alpha, \alpha')$ is a smooth manifold of dimension equal to $\ind(D_A)$.
	\item[(ii)] For all reducible flat connections $\Lambda$ on $W$, the perturbed ASD operator $D^\pi_\Lambda$ is injective.
	\end{itemize}
\end{lemma}

We call such a perturbation a {\it small regular perturbation}. 

\begin{proof}
As discussed above, perturbations satisfying (i) are generic. 

Because reducible connections on a cobordism with $b^+(W) = 0$ are cut out transversely in the reducible locus \cite[Lemma 4.20]{mme:eq-inst}, the perturbation $\pi(A)$ can be assumed to vanish when $A$ is reducible. The ASD operators coincide for any central connection (as the ASD operator only depends on the associated connection on the adjoint $SO(3)$-bundle). When $b_1(W) = b^+(W) = 0$, the ASD operator for the trivial connection has $D_\Lambda$ injective with cokernel of rank $3$, isomorphic to the cokernel of $d_\Lambda^*$ as $d_\Lambda^+$ is surjective. It follows that the same is true for any central connection. Because injectivity is an open condition, the same holds for any small perturbation $\pi$. 

That we may choose $\pi$ so that this is also true for abelian flat connections follows from the argument of \cite[Theorem 4.37]{mme:eq-inst}; see also \cite[Section 7.3]{CDX:surgery-n-gon} for a similar discussion and conclusion. For completeness, we will outline the argument here. 

Suppose $\Lambda$ is a flat abelian connection on $W$. As discussed above, $\ind(D_\Lambda) = -3$. Because $D_\Lambda$ is a $\Gamma_\Lambda \cong S^1$-equivariant Fredholm operator, this index breaks into a sum of two terms: the index internal to the reducible locus (which is $-1$ by Hodge theory) and the \emph{normal index} (the index of this operator normal to the reducible locus), which is thus $-2$. Write $D_{\pi,\Lambda}^\nu$ for this normal part.

While $\pi$ is trivial on the reducible locus, it is non-trivial normal to the reducible locus. For each flat reducible $\Lambda$, sending $\pi \mapsto D_{\pi,\Lambda}^\nu$ defines a map from the space of perturbations to the space of $S^1$-equivariant Fredholm operators of index $-2$. One argues using the fact that the normal index is non-positive that this map is transverse at $\pi = 0$ to the locus of operators with nonzero kernel (a union of countably many submanifolds, all of codimension at least $2$), and thus that for small generic $\pi$ the operator $D^\pi_\Lambda$ is injective.\footnote{If the normal index were non-negative, one could instead argue that generically the normal part of the perturbed ASD operator is surjective.}

The assumption that $b_1(W) = 0$ guarantees that there are only finitely many reducible flat connections, so that for generic small $\pi$ and all reducible flat $\Lambda$ the operator $D^\pi_\Lambda$ is injective.
\end{proof}

Supposing that $\pi$ is sufficiently small, the topological energy behaves as it does for unperturbed instantons:

\begin{lemma}\label{lemma:top-energy}
If $(W,\pi)$ is as in Lemma \ref{flat-W-pert}, and $\pi$ is sufficiently small, then all instantons on $W$ have non-negative topological energy. Further, if $A$ is a reducible solution to the perturbed ASD equations, then $\kappa(A) \in \frac 12 \Z$ is a non-negative half-integer, zero if and only if $A$ is flat.
\end{lemma}
\begin{proof}
For any connection $A$ we have 
	\begin{align*}
	  \kappa(A)&= \frac{1}{8\pi^2} \int_W \tr(F_A \wedge F_A) = \frac{1}{8\pi^2} \left(\|F_A\|_{L^2}^2 - 2\|F_A^+\|_{L^2}^2\right)
	\end{align*}
Thus for any solution of \eqref{per-ASD} we have $\kappa(A) = \frac{1}{8\pi^2}\left(\|F_A\|_{L^2}^2 -2\|\pi(A)\|_{L^2}^2\right)$.
	
By taking $\pi$ small enough, we can guarantee that $\kappa(A)$ is greater than $-\epsilon$ for any fixed positive constant $\epsilon$. The mod $\Z$ value of $\kappa(A)$ belongs to a fixed finite set, as the mod $\Z$ value is determined by the Chern--Simons values of the flat connections $A$ is asymptotic to, and each component of $\partial W$ supports finitely many gauge equivalence classes of flat connections. The first part of the claim follows as soon as $$\epsilon < \frac 12 \text{min}\{|\CS(\alpha) - \CS(\beta)| \mid \alpha \text{ flat on } Y, \; \beta \text{ flat on } Y', \; \CS(\alpha) \ne \CS(\beta)\}.$$
	
As for the second part of the claim, because the boundary of $W$ is a disjoint union of 2-torsion homology spheres, the limits of $A$ are central. The Chern--Simons invariant of a central connection is $0$ or $\frac 12$: if a central connection $\theta$ on a 3-manifold $M$ has holonomy representation $\varphi: \pi_1(M) \to \{\pm 1\}$, and $\tilde M$ is the double-cover associated to $\ker(\varphi)$, then $\theta$ lifts to the trivial connection on $\tilde M$ which has zero Chern--Simons invariant. Thus $2\CS(\theta) \equiv 0 \mod \Z$, and it then follows that $\kappa(A) \in \frac 12\Z$. Finally, because the perturbation vanishes on reducible connections, a perturbed reducible instanton is an unperturbed instanton, and unperturbed instantons are flat if and only if they have $\kappa(A) = 0$. 
\end{proof}

When $b_1(W) = b^+(W) = 0$, the discussion of \cite[Chapter 5.4]{donaldson-book} shows that moduli spaces of instantons are canonically oriented; for general cobordisms $W$ between rational homology spheres, one must choose an orientation on $H^0(W) \oplus H^1(W) \oplus H^+(W)$. 

Given a regular perturbation $\pi$ and a one-dimensional moduli space $M_\kappa^\pi(W; \alpha, \alpha')$, we can apply gluing theory to explicitly determine the oriented ends of this moduli space in terms of the moduli spaces of $W$, $Y$ and $Y'$. A careful description of these ends requires case analysis depending on the reducibility type of $\alpha, \alpha'$, and whether the instantons being `glued' are irreducible or reducible. To cut down on case analysis, we describe the gluing result only in the case of interest to us. 

Suppose $W: Y \to Y'$ is a cobordism, where $Y$ is a regular integer homology sphere and $Y'$ is a disjoint union of regular $2$-torsion homology spheres. In the next statement, we will write $\theta$ for the trivial connection on $Y$ and $\theta'$ for the trivial connection on each component of $Y'$.

Suppose $W$ has $b_1(W) = b^+(W) = 0$, and is equipped with a small regular perturbation $\pi$. Fix $\kappa \in (0, \frac 12)$ and an {\it irreducible} flat connection $\alpha$ on $Y$ such that the moduli space $M^\pi_\kappa(W; \alpha, \theta')$ is a 1-dimensional oriented smooth manifold.

\begin{lemma}\label{lemma:end-counting}
In the situation above, the moduli space $M^\pi_\kappa(W; \alpha, \theta')$ has finitely many ends. The oriented finite set of ends may be identified as the disjoint union of the following:

\begin{itemize}
\item[(i)] A disjoint union over all products $\breve M_{\kappa_1}(Y; \alpha, \beta) \times M^\pi_{\kappa_2}(W;\beta,\theta')$ in which $\beta$ is irreducible, $\kappa_1 + \kappa_2 = \kappa$, and both factors are zero-dimensional;
\item[(ii)] A disjoint union over all products $M^\pi_{\kappa_1}(W; \alpha, \beta' \sqcup \theta') \times \breve M_{\kappa_2}(N';\beta', \theta')$, where $N'$ is a connected component of $Y'$, $\beta'$ is an irreducible flat connection on $N'$, $\kappa_1 +\kappa_2 = \kappa$, and both factors are zero-dimensional;
\item[(iii)] A disjoint union over $|H_1(W)/i_*H_1(\partial W)|$ copies of $\breve M_{\kappa}(Y; \alpha, \theta)$, all with the same orientation.
\end{itemize}
\end{lemma}

\begin{proof}
Our goal is to construct a finite set of open embeddings $(0,\infty) \hookrightarrow M_\kappa(W; \alpha, \theta')$ with disjoint image, one for each element enumerated above, so that the complement of their union is compact.

The key is that this moduli space admits a compactification $M^+_\kappa(W; \alpha, \theta')$ by adding \emph{broken trajectories} and \emph{ideal instantons} (corresponding to Uhlenbeck bubbling limits), as in \cite[Proposition 5.5]{donaldson-book}. While Donaldson's interest is primarily in the case where the ends of $W$ are modelled on integer homology spheres, the standing assumption in \cite[Section 5.1]{donaldson-book} is merely that for each flat connection $\beta$ over one of the connected components $N$ of $\partial W$, we have $H^1(N;\beta) = 0$. This follows from our assumption that the boundary of $W$ is a disjoint union over \emph{regular} two-torsion homology spheres.

The compactification includes `weak limits' (Uhlenbeck bubbles) in which energy accumulates at $k \ge 1$ points in $W$, and the limiting (broken) instanton $A$ has $\kappa(A) \le \kappa - k$; but because $\kappa(A) \ge 0$ and we assume $\kappa < 1$, such bubble points do not appear.

The relevant limits take the following form: a sequence $A_n$ of instantons in $M^\pi_\kappa(W;\alpha, \theta')$ may converge to a sequence $(B_1, \cdots, B_n, A, C_1, \cdots, C_m)$, where each $B_i$ is an instanton on $\Bbb R \times Y$, each $C_j$ is an instanton over $\Bbb R \times Y'$ which is constant on all but one connected component, $A$ is an instanton over $W$, and the limit at $+\infty$ of each instanton agrees with the limit at $-\infty$ of the previous. We say that the flat limits at $+\infty$ of $B_1, \cdots, B_n$, and the flat limits at $-\infty$ of the non-constant components of $C_1, \cdots, C_{m}$ are `intermediate' flat connections. Write $r$ for the number of these which are reducible, hence central. Because no Uhlenbeck bubble point appears, \cite[Proposition 5.7]{donaldson-book} gives 
\begin{equation}\label{eq:index}\sum_{i=1}^n \ind(D_{B_i}) + \ind(D_A) + 3r + \sum_{j=1}^m \ind(D_{C_j}) = \ind(D_{A_n}) = 1.\end{equation} 
To see the appearance of $3r$ in Donaldson's formula, note that his $\text{rank}(H^0_\rho)$ is the dimension of the stabilizer of $\rho$; it is trivial when the intermediate flat connection $\rho$ is irreducible and is rank $3$ when the intermediate flat connection is reducible (hence central, because we assume $\partial W$ supports no abelian flat connections). Each $\ind(D_{B_i})$ and $\ind(D_{C_j})$ is at least $1$, as these compute the dimensions of nonempty smooth manifolds with a free $\Bbb R$-action, so we have 
\begin{equation}\label{eq:index-2}
1 \ge n + m + 3r + \ind(D_A). 
\end{equation} 
We will now split into two cases depending on whether $A$ is irreducible or reducible.

If $A$ is irreducible, we have $\ind(D_A) \ge 0$, as this computes the dimension of the nonempty smooth manifold of irreducible solutions to the perturbed ASD equations. It follows that $r = 0$ and $(n,m)$ is either $(1,0)$ or $(0,1)$, with $\ind(D_A) = 0$ and $\ind(D_B) = 1$ in the first case or $\ind(D_C) = 1$ in the second case, corresponding to cases (i) or (ii) in the enumeration. In these cases, \cite[Theorem 4.17]{donaldson-book} implies that a punctured neighborhood of $(B, A)$ or $(A,C)$ in the compactification may be identified with $(0, \infty)$. This gives precisely one end of $M^\pi_\kappa(W; \alpha, \theta')$ per such pair, as in parts (i) and (ii) of the enumeration. 

If $A$ is reducible, then by Lemma \ref{lemma:top-energy} $\kappa(A)$ must be zero and $A$ must be flat. To see this, first observe that by the same lemma, perturbed solutions to the ASD equations have non-negative topological energy. Secondly, topological energy is additive under gluing of instantons, so $$0 \le \kappa(A) \le \sum \kappa(B_i) + \kappa(A) + \sum \kappa(C_j) = \kappa < \frac 12,$$ the final inequality by assumption on $\kappa$. Thus $\kappa(A)$ is a half-integer with $0 \le \kappa(A) < \frac 12$, hence zero.

As discussed before the proof of Lemma \ref{flat-W-pert}, because $A$ is flat and central on the ends, we have $\ind(A) = -3$. Because the incoming end is a central flat connection, it is not $\alpha$, and so is an intermediate flat connection; thus $r \ge 1$ and $n \ge 1$. It follows from \eqref{eq:index-2} that $n = r = 1$ and $m = 0$. Further, because $Y$ is an integer homology sphere, $A$ is trivial at $-\infty$. Because $m = 0$, the asymptotic value of $A$ at $+\infty$ coincides with the asymptotic value of the connections $A_n$ at $+\infty$, so we also have that $A$ is trivial at $+\infty$. Therefore, the possible limits are reduced precisely to pairs $(B, \Lambda)$, where $B \in \breve M_\kappa(Y; \alpha, \theta)$ and $\Lambda \in M^\pi_{0}(W; \theta, \theta')$. Because this moduli space consists of connections of index $-3$, and the irreducible locus is cut out transversely, it contains no irreducible connections whatsoever. Further, because $\pi$ vanishes on the reducible locus, $M^\pi_0(W;\theta, \theta')$ coincides with the set of flat reducible connections which are trivial on the ends. 

When $\Lambda$ is central, the map $d_\Lambda^+$ is surjective by Hodge theory, so the discussion of \cite[Chapter 4.4.1]{donaldson-book} (which relies on the assumption that $d_\Lambda^+$ is surjective) applies, giving an identification of a punctured neighborhood of each $(B, \Lambda)$ in the compactification with an oriented copy of $(0, \infty)$. See also \cite[Example 4.8(i)]{Don:ori}, which implies that all of the ends with fixed $B$ are oriented the same way, independent of the central connection $\Lambda$.

When $\Lambda$ is abelian, the map $d_\Lambda^+$ is not surjective (it has a cokernel of rank $2$), and we are instead in the situation of obstructed gluing theory. Here we instead apply the argument of \cite[Example 4.8(ii)]{Don:ori}, which implies that this gluing procedure contributes \emph{two} ends, all oriented the same for fixed $B$ (and the same as in the previous paragraph). Briefly, these ends correspond to the zero set of a section of a line bundle over the space of gluing parameters $S^2$, and this line bundle has Euler class $2$; a similar discussion can be found in \cite[Section 6.1]{DME1}.

In total, these contribute $2a+z$ copies of $\breve M^+_\kappa(Y; \alpha, \theta)$, all with the same orientation, to the set of oriented ends. By Corollary \ref{lemma:count-reds}, this coincides with the stated cardinality.

By the Uhlenbeck compactness theorem, a sequence of instantons $A_n$ in $M_\kappa^\pi(W; \alpha, \theta')$ has a subsequence which converges to one of the previously enumerated broken trajectories, so eventually lies in one of the neighborhoods enumerated above. If $A_n$ lies wholly in the complement of these neighborhoods, it must have a convergent subsequence in $M_\kappa^\pi(W; \alpha, \theta')$ itself, as desired.
\end{proof}

\subsection{Detailed argument}\label{sec:details}
Our manifold $W$ will be a cobordism from $\PHS$ to a disjoint union of some number of copies of ${-}\PHS$ and $\pm O$. Notice that $H_1(P;\Z) = 0$ and $H_1(O; \Z) = (\Z/2)^2$, so these are two-torsion homology spheres. 

To apply the results of the previous section, we need to recall some facts about the moduli spaces of instantons on these 3-manifolds. 

\begin{lemma}\label{PO-basic-prop}
If $Y$ is a spherical 3-manifold, then $Y$ is regular with respect to the spherical metric. We also have the following facts about flat connections and instantons over $\Bbb R \times Y$ for $Y = \pm O, \pm \PHS$.
\begin{itemize}
\item[(i)] The manifold $\PHS$ supports three flat connections: the trivial connection $\theta$ and two irreducible flat connections $\alpha_1, \alpha_2$ with $\CS$-values $\frac{1}{120}, \frac{49}{120}$ respectively. The moduli space $\breve M_{1/120}(\PHS; \alpha_1, \theta)$ is a singleton. 
\item[(ii)] Every other nonempty moduli space $\breve M_\kappa(\PHS; \alpha, \beta)$ has $\kappa \ge 2/5$. 
\item[(iii)] Every nonempty moduli space $\breve M_\kappa({-}\PHS; \alpha, \theta)$ has $\kappa \ge \frac{71}{120}$.
\item[(iv)] Every nonempty moduli space $\breve M_\kappa(\pm O; \alpha, \beta)$ has $\kappa \ge \frac{1}{48}$.
\end{itemize}
\end{lemma}

\begin{proof}
Item (i) is the most substantial, and we discuss this last. The claim about smoothness is proved in \cite[Section 4.5]{austin}: for a spherical 3-manifold $Y = S^3/\Gamma$, if $\Bbb R \times Y$ is equipped with the product metric corresponding to the round metric on $Y$ and $A$ is a finite energy ASD connection with respect to this metric, then we may lift $A$ to $\Bbb R \times S^3$ and apply removablity of singularities to obtain an ASD connection $\widetilde A$ on $S^4$ with respect to the round metric. The cokernel of the ASD operator of $A$ is identified with the $\Gamma$-invariant part of the cokernel of the ASD operator of $\widetilde A$, which is zero for any ASD connection on $S^4$. 

As for items (ii)-(iv), a non-empty moduli space has $\kappa > 0$ (an instanton has $\kappa \ge 0$, but we define $\breve M_0(Y;\alpha, \beta)$ to be empty), so we have $\CS(\alpha)-\CS(\beta) \equiv \kappa > 0$. Items (ii) and (iii) follow from (i) by inspection of these Chern--Simons invariants, where for (iii) we use that $\CS_{{-}Y}(\alpha) = -\CS_Y(\alpha)$; item (iv) follows because $S^3$ is a 48-fold cover of $O$, so the minimal positive difference between Chern--Simons invariants is $\frac{1}{48}$.

We will establish item (i) in three steps. First, we enumerate the flat connections of $\PHS$ and their Chern--Simons values. Second, we show that there is a zero-dimensional moduli space containing a single ASD connection $A_0$ that is asymptotic to the trivial connection on the outgoing end. Finally, we establish that $A_0$ has energy $\frac{1}{120}$ and hence it is asymptotic to $\alpha_1$ on the incoming end. 

The first part will follow from the work of Fintushel--Stern \cite{FS:HFSF}. Their conventions are different from ours in several ways. First, they study the self-duality equations $F_A = *F_A$ over $Y \times \Bbb R = \Bbb R \times (-Y)$, whereas we study the anti-self-duality equations $F_A = -*F_A$ over $\Bbb R \times Y$. There is a canonical bijection between these two sets of connections. Next, Fintushel and Stern work with $SO(3)$-connections, whereas we work with $SU(2)$-connections, and their version of topological energy (called `Pontryagin charge' in \cite{FS:HFSF}\footnote{The expression for Pontryagin charge in \cite{FS:HFSF} is written with the incorrect sign; the authors intend the constant factor outside their integral to be $\frac{-1}{8\pi^2}$, as this would be the factor needed to make the first displayed equation in the proof of \cite[Theorem 3.9]{FS:HFSF} correct.}) is related to our version of topological energy by a constant factor of $-4$, related to the fact that if $g \in SU(2)$ and $\text{ad}_g \in SO(3)$, then $\text{tr}(\text{ad}_g^2) = 4\text{tr}(g^2)$. 

It is shown in \cite[Proposition 2.8]{FS:HFSF} that there is a bijection between the set of irreducible flat connections on $\Sigma(p,q,r)$ and the set of triples $(k,\ell,m)$ of integers $0 < k < p, \;\; 0 < \ell < q, \;\; 0 < m < r$ satisfying certain further conditions. Write $\alpha_{k,\ell,m}$ for the corresponding flat connection. In the special case $(p,q,r) = (2,3,5)$, there are two such triples: $(1,1,1)$ and $(1,1,3)$. We write the associated irreducible flat connections as $\alpha_1$ and $\alpha_2$, respectively. For a general irreducible flat connection $\alpha_{k,\ell,m}$ on $\Sigma(p,q,r)$, the computation of \cite[Theorem 3.7]{FS:HFSF} establishes
\[
  \CS(\alpha_{k,\ell,m}) = \frac{(kqr+\ell pr + mpq)^2}{4pqr} \mod \Z,
\]
by constructing a connection $A$ over $[0,1] \times \Sigma$ which is trivial at $0$ and equal to $\alpha_{k,\ell,m}$ at $1$, and establishing that for this connection $p_1(A) = \frac{(kqr+\ell pr + mpq)^2}{pqr}.$ The formula above follows from the relation $\kappa(A) = -\frac 14 p_1(A).$

The remaining parts follow from the work of \cite[Section 4.3]{austin}, but we refer to the more recent discussion of \cite[Section 4]{G:equi-poly-spaces}. There the author discusses \emph{framed moduli spaces} on $\Bbb R \times (-\PHS)$, denoted $\overline{\mathcal M}_z(-\PHS; \alpha, \beta)$, which carry a free action of $SO(3)$, a pair of equivariant maps to $SO(3)/\Gamma_\alpha$ and $SO(3)/\Gamma_\beta$ (where $\Gamma_\alpha$ is the set of $\alpha$-parallel gauge transformations). Here $z$ is the homotopy class of a path between $\alpha$ and $\beta$ in the configuration space of all connections on $\PHS$. The quotient of $\overline{\mathcal M}_z(-\PHS; \alpha, \beta)$ by the $SO(3)$ action is a compactification of our $\breve M_{\kappa(z)}(-\PHS; \alpha, \beta)$ for a choice of ${\kappa(z)}$ that is uniquely determined by $z$. When $\breve M_{\kappa(z)}(-\PHS; \alpha, \beta)$ is zero-dimensional, we simply have $\overline{\mathcal M}_z(-\PHS; \alpha, \beta)/SO(3) = \breve M_{\kappa(z)}(-\PHS; \alpha, \beta)$. 

First, \cite[Theorem 4.1]{G:equi-poly-spaces} asserts that when $\theta$ is `adjacent to $\beta$' in a certain graph $\mathcal S_{I^*}$, there is a choice of $z$ such that the moduli space $\overline{\mathcal M}_z(-\PHS; \theta, \beta)$ is diffeomorphic to $SU(2)/\Gamma'$ for some discrete subgroup $\Gamma' \subset SU(2)$. Because this space carries a free $SO(3)$-action, we must have $\overline{\mathcal M}_z(-\PHS; \theta, \beta) = SO(3)$ and thus $\Gamma' = \Z/2$. It follows immediately from the definition in \cite[Section 3.3]{G:equi-poly-spaces} that for $Q$ the `canonical representation' $\pi_1(\PHS) \to SU(2)$, we have $\theta$ adjacent to $Q$ in the graph $\mathcal S_{I^*}$. Thus for this irreducible representation $Q$, there is a zero-dimensional moduli space $\breve M_{\kappa(z)}(\PHS; Q, \theta)$ containing a single element $A_0$. 

We show that $\kappa(z)=\kappa(A_0)=\frac{1}{120}$, which in particular implies that $Q=\alpha_1$ using the computation of the Chern--Simons invariants of $P$. It is established in \cite[Equation (4.2)]{G:equi-poly-spaces} that the energy of this instanton coincides with $c_2(\widetilde E)/|\pi_1 \PHS|$, where $\widetilde E$ is the associated bundle over $S^4$. In \cite[Theorem 4.2]{G:equi-poly-spaces}, this Chern class is named $k$, and in \cite[Lemma 4.4]{G:equi-poly-spaces} a certain vector space $\mathscr H$ is introduced whose dimension is $k$. In \cite[Definition 4.20]{G:equi-poly-spaces} a constant $n_{\theta Q}$ is introduced which here is equal to $\dim \mathscr H$ (the group $\Gamma'$ was computed above to be $\Z/2$). Finally, in \cite[Lemma 4.23]{G:equi-poly-spaces} it is established that $n_{\theta Q} = 1$. Thus the corresponding instanton on $\Bbb R \times \PHS$ has energy $n_{\theta Q}/|\pi_1 \PHS| = \frac{1}{120}$.
\end{proof}

\begin{proof}[Proof of Proposition \ref{prop:CS-bound}]
As discussed earlier, if $X$ is positive-definite with boundary $m\PHS \# {-}kO$, we may construct a cobordism $W: \PHS \to \sqcup_{m-1} {-}\PHS \sqcup_{|k|} {\pm}O$ with $b_1(W) = b^+(W) = 0$. 
 
For a perturbation $\pi$ as in Lemma \ref{flat-W-pert}, we consider the moduli space $M_{1/120}^\pi(W; \alpha_1, \theta')$. Because $\alpha_1$ is irreducible, all elements of this moduli space are irreducible. Further, because one may obtain a connection $W: \alpha_1 \to \theta'$ of energy $\frac{1}{120}$ by gluing the instanton $A$ of Lemma \ref{PO-basic-prop}(i) to the trivial connection $\Theta$ on $W$ along the trivial connection $\theta$ on $\PHS$, we have by (\ref{eq:index}) that this moduli space has dimension 
$$\ind(D_B) + 3 + \ind(D_\Theta) = 1 + 3 - 3 = 1.$$ 
Here $\ind(D_B) = 1$ as $\ind(D_B)$ computes the dimension of the moduli space $M(\PHS; \alpha_1, \theta)$ before the quotient by the $\Bbb R$-action, and after quotienting this space is $0$-dimensional. So $M_{1/120}^\pi(W; \alpha_1, \theta')$ is a smooth oriented 1-manifold, and we may apply Lemma \ref{lemma:end-counting} to determine the ends of this moduli space. 

There are no ends of the form given in Lemma \ref{lemma:end-counting}(i), as such an instanton would factor through an instanton on $\Bbb R \times \PHS$ with $\kappa_1 \le \frac{1}{120}$ and irreducible flat limits $(\alpha, \beta)$, but no such instantons exist by Lemma \ref{PO-basic-prop}(ii).

There are also no ends of the form given in Lemma \ref{lemma:end-counting}(ii), as such an instanton would factor through an instanton on one of $\Bbb R \times {-}\PHS$ or $\Bbb R \times \pm O$ with $\kappa_2 \le \frac{1}{120}$, but no such instantons exist by Lemma \ref{PO-basic-prop}(iii)-(iv). 

Therefore all ends arise from Lemma \ref{lemma:end-counting}(iii), gluing the single instanton in $\breve M_{1/120}(\PHS; \alpha_1, \theta)$ to the reducibles on $W$ which are trivial on the ends; this contributes $|H_1(W)/i_* H_1(\partial W)|$ ends, all with the same sign. This contradicts the fact that an oriented 1-manifold with finitely many ends has zero ends, counted with sign.
\end{proof}

\bibliographystyle{alpha}
\bibliography{references}

\end{document}